\documentclass[11pt]{article}

\usepackage[letterpaper,margin=0.9in]{geometry}
\usepackage[utf8]{inputenc}
\usepackage[T1]{fontenc}
\usepackage{microtype}
\usepackage[numbers]{natbib}
\usepackage{hyperref}
\usepackage{booktabs}
\usepackage{nicefrac}

\usepackage{amsfonts}
\usepackage{amsmath}
\usepackage{amsthm}
\usepackage{amssymb}
\usepackage{dsfont}
\usepackage{comment}
\usepackage{multirow}
\usepackage{multicol}
\usepackage{threeparttable}
\usepackage{wrapfig}
\usepackage{xcolor}
\usepackage{color}
\usepackage{graphicx}
\usepackage{pifont}%
\newcommand{\algname}[1]{{\sf#1}\xspace}
\newcommand{\eqdef}{\coloneqq}

\usepackage{verbatim}
\usepackage{xspace} %

\usepackage{enumerate}
\usepackage{enumitem}

\usepackage{subcaption}

\providecommand{\algname}[1]{{\sf#1}\xspace}

\providecommand{\norm}[1]{\left\lVert#1\right\rVert}

  \providecommand{\R}{\mathbb{R}} %

  \providecommand{\Eb}[1]{{\mathbb E}\left[#1\right] }       %

  \DeclareMathOperator*{\argmin}{arg\,min}
  
  \DeclareMathOperator{\dom}{dom}

  \providecommand{\0}{\mathbf{0}}

  \providecommand{\dd}{\mathbf{d}}

  \renewcommand{\gg}{\mathbf{g}}

  \providecommand{\xx}{\mathbf{x}}
  \providecommand{\yy}{\mathbf{y}}

  \providecommand{\cO}{\mathcal{O}}

  \usepackage{bm}

\usepackage{cleveref}

\newtheorem{lemma}{Lemma}
\newtheorem{fact}{Fact}

\newtheorem{remark}{Remark}
\newtheorem{assumption}{Assumption}
\newtheorem{theorem}{Theorem}

\usepackage{url}

\renewcommand{\epsilon}{\varepsilon}

\usepackage{algorithm}
\usepackage[noend]{algpseudocode}

\errorcontextlines\maxdimen

\makeatletter
    \newcommand*{\algrule}[1][\algorithmicindent]{\makebox[#1][l]{\hspace*{.5em}\thealgruleextra\vrule height \thealgruleheight depth \thealgruledepth}}%
\newcommand*{\thealgruleextra}{}
\newcommand*{\thealgruleheight}{.75\baselineskip}
\newcommand*{\thealgruledepth}{.25\baselineskip}

\newcount\ALG@printindent@tempcnta
\def\ALG@printindent{%
    \ifnum \theALG@nested>0%
        \ifx\ALG@text\ALG@x@notext%
        \else
            \unskip
            \addvspace{-1pt}%
            \ALG@printindent@tempcnta=1
            \loop
                \algrule[\csname ALG@ind@\the\ALG@printindent@tempcnta\endcsname]%
                \advance \ALG@printindent@tempcnta 1
            \ifnum \ALG@printindent@tempcnta<\numexpr\theALG@nested+1\relax%
            \repeat
        \fi
    \fi
    }%
\usepackage{etoolbox}
\patchcmd{\ALG@doentity}{\noindent\hskip\ALG@tlm}{\ALG@printindent}{}{\errmessage{failed to patch}}
\makeatother

\newbox\statebox
\newcommand{\myState}[1]{%
    \setbox\statebox=\vbox{#1}%
    \edef\thealgruleheight{\dimexpr \the\ht\statebox+1pt\relax}%
    \edef\thealgruledepth{\dimexpr \the\dp\statebox+1pt\relax}%
    \ifdim\thealgruleheight<.75\baselineskip
        \def\thealgruleheight{\dimexpr .75\baselineskip+1pt\relax}%
    \fi
    \ifdim\thealgruledepth<.25\baselineskip
        \def\thealgruledepth{\dimexpr .25\baselineskip+1pt\relax}%
    \fi
    \State #1%
    \def\thealgruleheight{\dimexpr .75\baselineskip+1pt\relax}%
    \def\thealgruledepth{\dimexpr .25\baselineskip+1pt\relax}%
}

\usepackage{mathtools}
\DeclarePairedDelimiterX{\inp}[2]{\langle}{\rangle}{#1, #2}
\DeclarePairedDelimiterX{\cbr}[1]{\{}{\}}{#1} %
\DeclarePairedDelimiterX{\rbr}[1]{(}{)}{#1} %
\DeclarePairedDelimiterX{\sbr}[1]{[}{]}{#1} %

\usepackage{thm-restate}

\crefname{fact}{Fact}{Facts}
\Crefname{fact}{Fact}{Facts}
\crefname{assumption}{Assumption}{Assumptions}
\Crefname{assumption}{Assumption}{Assumptions}

\usepackage{import}

\title{The Dual Averaging Power-Prox Method \\with Application to Heavy-Tail Incremental Gradient}

\author{%
  Yuan Gao\thanks{CISPA Helmholtz Center for Information Security, Germany. \texttt{\{yuan.gao, michael.rack, stich\}@cispa.de}.} \thanks{Universit\"at des Saarlandes, Germany.}
  \and
  Jeremy Rack\footnotemark[1] \footnotemark[2]
  \and
  Sebastian U. Stich\footnotemark[1]
}

\date{}

\begin{document}

\maketitle

\begin{abstract}
We study finite-sum composite optimization under two departures from classical stochastic gradient descent theory that are central in practice: incremental gradient access and heavy-tailed gradient noise. Specifically, we consider fixed cyclic passes over component gradients and assume that, at the optimum, component gradients have a bounded $q$-th centralized moment for some $q\in(1,2]$. This setting is much closer to modern ML training practice than the assumptions used in classical \algname{SGD} theory, yet its theoretical understanding remains limited. We propose a Dual Averaging Power-Prox method for incremental gradients and establish, to the best of our knowledge, the first convergence analysis in this regime. We further show that our method achieves a better asymptotic convergence rate than the corresponding \algname{SGD} method with i.i.d. (with-replacement) sampling.
\end{abstract}

\section{Introduction}
\label{sec:introduction}
Solving the following finite-sum composite optimization problem is one of the most fundamental tasks in machine learning and optimization (see \Cref{sec:problem-formulation} for more precise definitions and assumptions):
\[
    \min_{\xx \in \dom\psi} F(\xx) \eqdef f(\xx) + \psi(\xx), \quad \text{where} \quad f(\xx) =\frac{1}{n} \sum_{i=0}^{n-1} f_i(\xx).
\]
Problems of this form are ubiquitous in machine learning, as they cover a wide range of applications, including regularized learning problems~\citep{Liu_2015_CVPR}, signal processing~\citep{combettes2011proximal}, image processing~\citep{Luke2020}, and many others. They also naturally cover constrained optimization by using the indicator function of the constraint set. Recent variants have been applied to distributed and federated learning as well~\citep{mishchenko2022proxskip}. 

In modern ML training, $n$ represents the number of component functions and is often extremely large. Computing the full gradient of $f$ therefore becomes prohibitively expensive. As a result, stochastic gradient descent (\algname{SGD}), which computes one component gradient per iteration, has become the workhorse of modern ML. Classical \algname{SGD} theory relies on a few fundamental assumptions on component-gradient access and gradient noise~\citep{ghadimi2016mini,ghadimi2013accelerated}. Namely,
\begin{itemize}
    \item[(A1)] \textbf{Unbiasedness and independence:} The stochastic gradient at iteration $t$ is an unbiased estimator of the full gradient, i.e., $\Eb{\gg_t} = \nabla f(\xx_t)$, and is independent of the past, i.e., the randomness in $\gg_t$ is independent of all iterations up to $\xx_t$.
    \item[(A2)] \textbf{Bounded variance:} The \emph{variance} of the stochastic gradient is bounded, i.e., $\Eb{\norm{\gg_t - \nabla f(\xx_t)}^2} \leq \sigma^2$.   
\end{itemize}
While a vast literature has been built on (A1) and (A2), these assumptions deviate significantly from modern practice. Accordingly, recent work has focused on understanding gradient methods that relax either (A1) or (A2).

Regarding (A1), almost all modern ML training algorithms access component gradients in an \emph{epoch-wise} manner, i.e., sequentially within each epoch. This access pattern provides better cache locality and often better empirical performance than i.i.d. sampling. It is the de facto standard in modern ML frameworks and its variants are widely used to train deep neural networks, in particular large language models~\citep{bengio2012practical,chowdhery2023palm,touvron2023llama}. 

From a theoretical perspective, epoch-wise access creates strong dependence across iterations, so the gradient at a given iteration is no longer unbiased. One can even fix the epoch permutation, in which case there is no randomness at all in gradient access. This pattern is known as \emph{incremental gradient}. Incremental gradient methods have been studied extensively in optimization, dating back at least to \citet{nedic2001convergence}, and have gained increasing attention in recent years due to their close relevance to practice~\citep{mishchenko2022proximal,mishchenko2020random,liu2024last,koloskova2024on}. 

Regarding (A2), many recent works study extensions and relaxations of the bounded-variance assumption~\citep{khaledbetter,alacaoglu2025towards}. Most notably, there has been a surge of interest in gradient methods under heavy-tailed noise, i.e., when stochastic gradients have a bounded $q$-th centralized moment for some $q\in (1,2]$~\citep{vural2022mirror,fatkhullin2025can,hubler2025gradient,fradintight}. This direction is strongly motivated by empirical evidence that gradients in many ML workloads are heavy-tailed, including image classification~\citep{simsekli2019tail,battash2024revisiting} and policy-gradient optimization in reinforcement learning~\citep{garg2021proximal}. Since heavy-tailed noise can imply unbounded variance, vanilla \algname{SGD} may diverge. Recent work either proposes \algname{SGD} variants robust to heavy-tailed noise or establishes lower bounds for such methods.

Despite extensive progress on each axis individually, there is very little work that treats both simultaneously, i.e., incremental gradient methods under heavy-tailed noise. Yet this is arguably the most natural and relevant setting for modern ML training. In this work, we fill precisely this gap.

\subsection{Our contributions}
As discussed above, incremental gradient methods and heavy-tailed noise are both fundamental in modern ML training, but they are mostly understood in isolation. This leaves a clear theoretical gap: understanding incremental gradient methods under heavy-tailed noise. In this work, we:

\begin{itemize}
    \item We introduce the dual-averaging power-prox framework. The method combines dual averaging with a $p$-th power-proximal regularizer, where $p$ is conjugate to the heavy-tail moment parameter $q$. This design is tailored to bounded $q$-th moment assumptions and provides the algorithmic basis for both the SGD and incremental-gradient methods studied in the paper. We analyze the dual-averaging power-prox method both when gradient access is stochastic and i.i.d. and when gradient access is incremental. To the best of our knowledge, this is the first analysis of an incremental method under heavy-tailed noise. We show a better asymptotic convergence rate than the corresponding \algname{SGD} method.
    \item We further show that dual averaging addresses a common restriction in analyzing incremental methods~\citep{mishchenko2022proximal,liu2024last}: proximal operations can only be performed at the end of epochs. Our method and analysis allow proximal updates at every iteration, which may be of independent interest.
    \item We also corroborate our theoretical findings with numerical experiments.
\end{itemize}

\subsection{Related work}
There is a large body of work on incremental gradient methods and on optimization under heavy-tailed noise. We briefly review the most relevant results.
\paragraph{Incremental gradient methods} Theoretical analysis of incremental gradient methods (more generally, shuffling-style methods) is a long-standing topic in optimization, dating back at least to \citet{nedic2001convergence}. The work of \citet{gurbuzbalaban2021random} was among the first to establish that random-shuffling methods can admit better guarantees than standard \algname{SGD} in strongly convex settings. Subsequent works extended these results to broader settings and simplified parts of the analysis, e.g., \citet{haochen2019random,nguyen2021unified,mishchenko2020random,liu2024last}; much of this literature focuses on asymptotic rates. Recent works also improved non-dominant terms, for example by removing epoch-size dependence in the stepsize for non-convex non-composite problems~\citep{koloskova2024convergence}, or by using less restrictive notions of component smoothness~\citep{cai2023empirical}. For composite optimization, incremental analysis poses additional technical challenges: many recent results analyze variants where proximal updates are applied only at epoch boundaries~\citep{mishchenko2022proximal,liu2024last,josz2024proximal}. In contrast, our analysis shows that dual averaging can handle this issue and hence allows proximal updates at every iteration.

\paragraph{Heavy-tailed noise} Recent work on heavy-tailed noise largely focuses on \algname{Clip-SGD} and adaptive clipping variants~\citep{liu2025clipped,chezhegov2025clipping,fradintight}, as well as \algname{Normalized-SGD} and momentum-based variants~\citep{liunonconvex,fatkhullin2025can,hubler2025gradient,fradintight}. Other lines of work study sign-based \algname{SGD} methods~\citep{yu2026sign} and more general nonlinear gradient transformations~\citep{armacki2024nonlinear,armacki2025optimal}. The work most closely related to ours is \citet{vural2022mirror}, which develops order-$p$ uniformly convex mirror maps for heavy-tailed settings (and can be viewed as a modern treatment of the ideas of mirror descent~\citep{nemirovskij1983problem}). Our power-prox technique is a specific instantiation of this framework. \citet{vural2022mirror} does not analyze incremental-gradient sampling nor the dual-averaging variant; we discuss the \algname{SGD} dual averaging variant in \Cref{sec:warmup} as a warm-up to our main method.

\section{Problem formulation, assumptions, and notation}
\label{sec:problem-formulation}

We consider the finite-sum optimization problem of the following form:
\begin{equation}
    \label{eq:problem}
    \min_{\xx \in \dom\psi} F(\xx) \eqdef f(\xx) + \psi(\xx), \quad \text{where} \quad f(\xx) = \frac{1}{n}\sum_{i=0}^{n-1} f_i(\xx),
\end{equation}
where $\xx\in \R^d$ are the parameters of the model that we train, and $F:\dom \psi\subset \R^d\to \R$ is the objective. We assume that the problem has a solution, denoted by $\xx^\star\in \dom \psi$, and we write $F^\star \eqdef F(\xx^\star)$. The objective $F$ is a composite function consisting of a smooth part $f$ and a composite part $\psi$. The smooth part $f$ is a finite sum of $n$ component functions $f_i$. The composite part $\psi$ is a simple convex function for which the associated proximal subproblems are easy to solve. We denote by $\dom\psi$ the set where $\psi$ is finite. This formulation naturally encodes constrained optimization problems by letting $\psi$ be the indicator function of the constraint set. We make the following assumptions on the problem:

\begin{assumption}
\label{assumption:convexity}
We assume that $f, f_i$ are convex, and $\psi$ is convex, closed, and proper over the convex domain $\dom\psi$. Moreover, we assume that each $f_i$ is differentiable. %
\end{assumption}
Next we assume that each $f_i$ is $L$-smooth, which is standard in the literature of incremental gradient methods~\citep{mishchenko2022proximal,mishchenko2020random,liu2024last,koloskova2024on}. 
\begin{assumption}
    \label{assumption:smoothness}
    We assume that each $f_i$ is $L$-smooth, i.e., for all $\xx, \yy \in \R^d$,
    \[
        \norm{\nabla f_i(\xx) - \nabla f_i(\yy)} \leq L \norm{\xx - \yy}.
    \]
\end{assumption}
Next we make the following heavy-tail assumption on the level of deviations of the component gradients $\nabla f_i$ from the full gradient $\nabla f$, at the \emph{optimum}. This is more commonly known as the bounded $q$-th centralized moment assumption~\footnote{A more precise term here is $q$-th order component-gradient heterogeneity, i.e., the average $q$-th power deviation of component gradients from the full gradient at the optimum. This quantity equals the $q$-th centralized moment when component gradients are sampled uniformly at random, so we follow existing literature and use probabilistic terminology, though our main setting, incremental gradient, does not involve any randomness.} in the stochastic optimization literature.
\begin{assumption}
    \label{assumption:heavy-tail}
    We assume that there exist a constant $q\in (1,2]$ and $\sigma_\star^q$ such that %
    \[
        \frac{1}{n} \sum_{i=0}^{n-1} \norm{\nabla f_i(\xx^\star) - \nabla f(\xx^\star)}^q \leq \sigma_\star^q.
    \]
\end{assumption}
When $q=2$, this recovers the standard bounded variance assumption. Throughout the paper, we always write $p\geq 2$ as the conjugate exponent of $q$, i.e., $\frac{1}{p} + \frac{1}{q} = 1$.

We always write $\beta_f(\xx,\yy)$ to denote the Bregman divergence of $f$ between $\xx$ and $\yy$, centered at $\xx$, i.e.
\[
    \beta_f(\xx,\yy) \eqdef f(\yy) - f(\xx) - \inp{\nabla f(\xx)}{\yy-\xx}.
\]

\section{Warmup: dual averaging power-prox for \algname{SGD}}
\label{sec:warmup}
\begin{algorithm}[tb]
    \caption{Dual Averaging Power-Prox for \algname{SGD}}
    \label{alg:da-power-prox-sgd}
    \begin{algorithmic}[1]
        \State \textbf{Input:} $\xx_0$ and $\gamma,\lambda\in\R_+$. 
        \For{$t = 0,1,\dots$}
            \State Obtain $\gg_t \eqdef \nabla f_{i_t}(\xx_t)$, where $i_t$ is independently sampled from $\{0,\ldots,n-1\}$ uniformly at random.
            \State $\xx_{t+1} =\argmin_{\xx}\left[ 
                \sum_{s=0}^{t} \big(f(\xx_s)+\inp{\gg_s}{\xx-\xx_s}+\psi(\xx)\big)
                +
                \frac{\gamma}{2} \norm{\xx - \xx_0}^2+\frac{\lambda}{p}\norm{\xx-\xx_0}^p\right],
                $
                \Statex \hspace{\algorithmicindent}where $\frac{1}{p}+\frac{1}{q} = 1$.
        \hfill
        \EndFor
    \end{algorithmic}
\end{algorithm}
Before we delve into our main method for incremental gradient methods, we first discuss the power-prox method for the basic \algname{SGD} method, where the (stochastic) gradient at each iteration is sampled among all $n$ component gradients uniformly at random and independently from the past. In \Cref{alg:da-power-prox-sgd}, we present the dual-averaging power-prox method in this simpler i.i.d. sampling setting. This dual-averaging formulation is distinct from the mirror-descent power-prox method studied by \citet{vural2022mirror}, and it will serve as the basis for our incremental-gradient method in \Cref{sec:incremental}.

The advantage of applying the dual averaging technique will become clearer in the next sections when we discuss the additional technical challenges with incremental gradients. Compared to the usual dual averaging method, \Cref{alg:da-power-prox-sgd} has two regularizers in the subproblem, namely the usual order~$2$ regularizer and an additional order~$p$ regularizer, where $p$ is the conjugate of $q$ in \Cref{assumption:heavy-tail}. Hence the name ``power-prox''. Intuitively speaking, the order~$2$ term provides the regularization with respect to the optimization landscape (i.e. the smoothness of the problem). The order $p$ term provides the regularization with respect to the noise. The gradient noise lies in the dual space and has a bounded $q$-th centralized moment; therefore, in the primal space, the regularization should be of order $p$. More concretely, the power-prox term induces the order $p$ uniform convexity terms in the descent analysis and leads to:

\begin{align*}
    & \sum_{t=0}^{T-1}\Eb{F(\xx_{t+1})-F^\star} + \frac{\gamma}{2}\sum_{t=0}^{T-1}\Eb{\norm{\xx_{t+1}-\xx_t}^2} + \frac{\lambda}{p2^{p-2}} \sum_{t=0}^{T-1}\Eb{\norm{\xx_{t+1}-\xx_t}^p}\\
    \leq & \frac{\gamma R_0^2}{2} + \frac{\lambda R_0^p}{p} + \sum_{t=0}^{T-1}\Eb{\inp{\gg_t-\nabla f(\xx_t)}{\xx_t-\xx_{t+1}}},
\end{align*}
where on the LHS the $\frac{\lambda}{p2^{p-2}}\Eb{\norm{\xx_{t+1}-\xx_t}^p}$ term comes from the order $p$ uniform convexity. On the RHS, the residual error $\Eb{\inp{\gg_t-\nabla f(\xx_t)}{\xx_t-\xx_{t+1}}}$ should clearly be separated into one noise part and one primal iterate distance part. Since \Cref{assumption:heavy-tail} only gives us a bound on the $q$-th moment of the noise, splitting the residual using Young's inequality naturally leads to $\Eb{\norm{\xx_{t+1}-\xx_t}^p}$, which matches the order $p$ uniform convexity term on the LHS and can be absorbed. This is the key intuition behind the power-prox term.

\begin{theorem}
    \label{thm:da-power-prox-sgd}
    Under \Cref{assumption:convexity,assumption:smoothness,assumption:heavy-tail}, if we choose $\gamma = \max\{8L, \sqrt{\frac{8L(F(\xx_0)-F^\star)}{R_0^2}}\}$, and $\lambda = \left(\frac{p\sigma_\star^qT}{2qR_0^p}\right)^{\frac{1}{q}}$ where $R_0\eqdef \norm{\xx^\star-\xx_0}$, then it takes at most
    \[
        T = \frac{16LR_0^2 + 4\sqrt{2L(F(\xx_0)-F^\star)R_0^2}}{\varepsilon} + \frac{16^{\frac{q}{q-1}}\sigma_\star^{\frac{q}{q-1}}R_0^{\frac{q}{q-1}}}{\varepsilon^{\frac{q}{q-1}}},
    \]
    iterations of \Cref{alg:da-power-prox-sgd} to achieve $\frac{1}{T}\sum_{t=0}^{T-1}  \Eb{F(\xx_{t+1}) - F(\xx^\star)} \leq \varepsilon$.
\end{theorem}
    We note that the $\frac{4\sqrt{2L(F(\xx_0)-F^\star)R_0^2}}{\varepsilon}$ term comes from the composite nature of the problem. It does not affect the asymptotic rate, and can be removed by an initialization step (we refer to \Cref{remark:main-convergence} for more details).

    Interestingly, in the stepsize selection here, we see a clear distinction between the roles of $\gamma$ and $\lambda$. $\gamma$ is selected to be the smoothness parameter $L$ (note that we can essentially bound $F(\xx_0)-F^\star$ via $LR_0^2$; see more discussion in \Cref{remark:main-convergence}), and is therefore only related to the smoothness of the problem and the optimization term $\cO(\frac{LR_0^2}{\varepsilon})$. $\lambda$, however, is determined by the noise level $\sigma_\star^q$, the initial distance $R_0$, and the number of iterations $T$. This is consistent with our intuition that the power-prox term is designed to directly regularize the noise, and therefore its strength should be determined by the noise level and the number of iterations. We will see in the next section how this distinction between the roles of $\gamma$ and $\lambda$ is challenged in the incremental gradient setting, which is crucial for showing that incremental gradient methods can attain better asymptotic rates than the \algname{SGD} method.

\section{Dual averaging power-prox for incremental gradients}
\label{sec:incremental}
\begin{algorithm}[tb]
    \caption{Dual Averaging Power-Prox for Incremental Gradient}
    \label{alg:da-power-prox-incremental}
    \begin{algorithmic}[1]
        \State \textbf{Input:} $\xx_0$ and $\gamma,\lambda\in\R_+$. 
        \For{$k = 0,1,\dots, K-1$ }
            \State $\xx^k\eqdef \xx_{kn}$
            \For{$i = 0,1,\dots, n-1$}
                \State $t = k \cdot n + i$
                \State Obtain $\gg_k^i \eqdef \gg_t = \nabla f_i(\xx_t)$. 
                \State $\xx_{k}^{i+1}\eqdef \xx_{t+1} =\argmin_{\xx}\Phi_k^i(\xx)$, where
                \State $\Phi_k^i(\xx)\eqdef
                  \begin{aligned}[t]
                  &\sum_{e=0}^{k-1}\sum_{j=0}^{n-1} \big(f(\xx^e)+ \inp{\gg_e^j}{\xx-\xx^e}+\psi(\xx)\big)  + \sum_{j=0}^{i} \big(f(\xx^k)+ \inp{\gg_k^j}{\xx-\xx^k}+\psi(\xx)\big) \\
                  &\quad + \frac{\gamma}{2} \norm{\xx - \xx_0}^2  + \frac{\lambda}{p} \norm{\xx-\xx_0}^p
                  \end{aligned}$
				\EndFor
		\EndFor
    \end{algorithmic}
\end{algorithm}
In this section, we consider the main setting of this paper, where the gradient at each iteration is obtained via a pass over the $n$ component functions in a fixed order, and the gradients of the component functions have a bounded $q$-th centralized moment at the optimum. We now develop the main algorithm of the paper: a dual-averaging power-prox method for incremental-gradient access, which we summarize in \Cref{alg:da-power-prox-incremental}. The algorithm runs in $K$ epochs, and each epoch consists of $n$ iterations. To facilitate understanding and analysis of the algorithm, we write $\xx^k=\xx_{kn}$ for the starting point of the $k$-th epoch, and $\xx_k^i$ for the $i$-th iterate within the $k$-th epoch, so that $\xx_k^0 = \xx^k$ and $\xx_k^n = \xx^{k+1}$. 

The algorithm maintains a sum of the linearized losses of all past epochs and the current epoch, and applies a proximal update at every iteration. The prox function is the same as in \Cref{alg:da-power-prox-sgd}, which consists of a quadratic term and a $p$-th power term. We note that in \Cref{alg:da-power-prox-incremental} we add the anchor points $\xx^e$ to the linearized losses and keep the (fixed) objective values $f(\xx^e)$. This gives us a more convenient form for the analysis but is irrelevant to the implementation and the actual updates of the algorithm. It is equivalent to solving the following subproblems at each iteration (we refer to \Cref{sec:implementation} for a brief discussion on solving the subproblems):
\begin{equation}
    \label{eq:subproblem-specific}
    \xx_{k}^{i+1}\eqdef \argmin_{\xx}\left[
                  \inp{\sum_{e=0}^{k-1}\sum_{j=0}^{n-1} \gg_e^j + \sum_{j=0}^{i} \gg_k^j}{\xx} + (kn+i+1) \psi(\xx) + \frac{\gamma}{2} \norm{\xx - \xx_0}^2  + \frac{\lambda}{p} \norm{\xx-\xx_0}^p
                  \right]
\end{equation}
In addition to the power-prox term which we discussed in the previous section, a key design choice of our algorithm is the use of the dual averaging technique. A key drawback in many of the earlier works on incremental gradient methods (for composite optimization) \citep{liu2024last,mishchenko2022proximal} is that they apply the proximal operation at the end of every epoch, and during one epoch the algorithm simply performs gradient steps without proximal operations. Such a design choice is not ideal, and somewhat unnatural, since the iterates can potentially step outside of the feasible region $\dom\psi$ during the epoch. This risks having iterates with potentially infinite objective values, or iterates for which the gradients are not even well-defined. 

At its core, maintaining simple intra-epoch gradient steps is a technical choice that makes the analysis work. The canonical way to analyze the convergence of incremental gradient methods (or shuffling methods in general) studies the progress of the algorithm in terms of the epoch-end iterates $\xx^k$. It relies on a joint analysis of the epoch-wide deviations of the iterates $\xx_k^i$ and the gradients $\gg_k^i$. This typically requires additive gradient accumulation within each epoch, which appears to conflict with performing proximal corrections at every iteration. Dual averaging is precisely what resolves this tension. It decouples \emph{how} gradients are accumulated from \emph{when} proximal updates are applied: gradients remain additive in the dual variable, while primal iterates are updated proximally at every iteration. As a result, our method keeps the iterate sequence in $\dom\psi$ throughout training without sacrificing the additive structure needed for epoch-wise reasoning. We also note that even when $\psi\equiv 0$, dual averaging is still crucial for us because our method applies the power-prox technique, which makes the updates naturally non-additive.

\subsection{Sketch of convergence analysis}
\label{sec:sketch-convergence-analysis}
Now we proceed to sketch the convergence analysis of \Cref{alg:da-power-prox-incremental}. We will present some of the key technical lemmas and the main convergence theorem, sketch some of the key proof insights, and defer the full proofs to \Cref{sec:missing-proofs-incremental}. We first give a descent statement in terms of the epoch-end iterates $\xx^k$, up to some error terms that are epoch-wise analogues of the typical (stochastic) dual averaging error terms.

\begin{restatable}{lemma}{EpochDescent}
    \label{lem:epoch-descent}
    Under \Cref{assumption:smoothness}, the iterates of \Cref{alg:da-power-prox-incremental} satisfies:
    \begin{equation}
        \label{eq:epoch-descent}
        \begin{aligned}
            &n\sum_{k=0}^{K-1} \big( F(\xx^{k+1}) -F^\star\big)+ \frac{\gamma-nL}{2}\sum_{k=0}^{K-1}\norm{\xx^{k+1}-\xx^k}^2 + \frac{\lambda}{p2^{p-2}}\sum_{k=0}^{K-1}\norm{\xx^{k+1}-\xx^k}^p\\
            \leq & \frac{\gamma R_0^2}{2} + \frac{\lambda R_0^p}{p} - \frac{\gamma R_K^2}{2} - \frac{\lambda R_K^p}{p2^{p-2}} + \sum_{k=0}^{K-1}\sum_{i=0}^{n-1}\big( \inp{\gg_k^i - \nabla f(\xx^k)}{\xx^\star-\xx^{k+1}}-\beta_f(\xx^k,\xx^\star)\big),
        \end{aligned}
    \end{equation} %
    where we write $R_0\eqdef \norm{\xx_0 - \xx^\star}$ and $R_K\eqdef \norm{\xx^K - \xx^\star}$. 
\end{restatable}
The above lemma is rather generic, in the sense that the resulting error terms $\inp{\gg_k^i - \nabla f(\xx^k)}{\xx^\star-\xx^{k+1}}-\beta_f(\xx^k,\xx^\star)$ are epoch-wise analogues of the typical error terms in the analysis of (stochastic) dual averaging methods, and we have not invoked any specific assumptions or algorithmic properties to bound them yet. We note that when proving the above lemma, we crucially preserved the Bregman divergence terms $\beta_f(\xx^k,\xx)$, instead of simply bounding them from below by $0$ using convexity as in most dual averaging analyses. This helps us control the error terms in a more fine-grained way and enables delicate epoch-wise cancellations that are crucial for the final convergence rates. We now upper bound the error terms via an epoch-wise drift bound $\sum_{i=0}^{n-1}\norm{\xx_k^i-\xx^k}^2$. Importantly, the upper bound does not involve any $\sigma_\star^q$ terms. These terms will only appear when we further upper bound the epoch-wise drifts.

\begin{restatable}{lemma}{ErrorBound}
    \label{lem:error-bound}
    Given \Cref{assumption:convexity,assumption:smoothness}, for any epoch $k$, we have:
    \begin{equation}
        \label{eq:error-bound}
        \sum_{i=0}^{n-1}\big( \inp{\gg_k^i - \nabla f(\xx^k)}{\xx^\star-\xx^{k+1}}-\beta_f(\xx^k,\xx^\star)\big) \leq \frac{Ln}{2}\norm{\xx^{k+1}-\xx^k}^2 + \frac{3L}{2}\sum_{i=0}^{n-1}\norm{\xx_k^i-\xx^k}^2.
    \end{equation} %
\end{restatable}
Since $\norm{\xx^{k+1}-\xx^k}^2$ can be easily canceled with the left-hand side of \Cref{eq:epoch-descent}, the main technical challenge that remains is to upper bound the epoch-wise drifts $\sum_{i=0}^{n-1}\norm{\xx_k^i-\xx^k}^2$. This is where our power-prox term comes into play. We will see that the power-prox term allows us to specifically regulate the noise induced by the gradients. 

\begin{restatable}{lemma}{DriftBound}
    \label{lem:drift-bound}
    Given \Cref{assumption:convexity,assumption:smoothness,assumption:heavy-tail}, assuming that $\gamma \geq 4nL$, we have:
    \begin{equation}
        \label{eq:drift-bound}
        \sum_{i=0}^{n-1}\norm{\xx_k^i-\xx^k}^2 \leq \frac{2n^2}{\gamma}\big(F(\xx^k)-F^\star\big) + \frac{2n^{q+1}\sigma_\star^q}{q\gamma\lambda^{q-1}}
    \end{equation}
\end{restatable}

\begin{remark}
    Importantly, since the noise $\sigma_\star^q$ is introduced only via the epoch-wise drifts, it is additionally regularized by a $\nicefrac{1}{\gamma}$ factor. This demonstrates a crucial difference between the usual \algname{SGD} and the incremental gradient method: for \algname{SGD}, the noise level is associated only with $\lambda$, the power-prox parameter, while for incremental gradients, both $\gamma$ and $\lambda$ play a role in regularizing the noise, which results in better asymptotic rates.
\end{remark}

\begin{proof}[Proof sketch]

    Via a telescoping argument, we can get:
    \begin{align*}
        \frac{2\gamma -iL}{2} \norm{\xx_k^i-\xx^k}^2 + \frac{\lambda}{p2^{p-1}}\norm{\xx_k^i-\xx^k}^p &\leq i\big(F(\xx^k)-F(\xx_k^i)\big) +  \sum_{j=0}^{i-1} \inp{\gg_k^j - \nabla f(\xx^k)}{\xx^k-\xx_k^i}\\
        &\leq i\big(F(\xx^k)-F^\star\big) +  \sum_{j=0}^{i-1} \inp{\gg_k^j - \nabla f(\xx^k)}{\xx^k-\xx_k^i}
    \end{align*}
    Therefore, we need to upper bound $\sum_{j=0}^{i-1} \inp{\gg_k^j - \nabla f(\xx^k)}{\xx^k-\xx_k^i}$. This is where it becomes clear why the power-prox term is crucial. It is clear that the gradient error $\gg_k^j - \nabla f(\xx^k)$, after some possible manipulations, has to be bounded via some Cauchy-Schwarz type inequalities. Since the noise is heavy-tailed with power $q$, Cauchy-Schwarz will raise the drifts $\norm{\xx_k^i-\xx^k}$ to the conjugate power $p$. This will then be nicely canceled with the power-prox induced terms. We have:
    \begin{align*}
        \inp{\gg_k^j - \nabla f(\xx^k)}{\xx^k-\xx_k^i} &= \inp{\nabla f_j(\xx_k^j) -\nabla f_j(\xx^k) +\nabla f_j(\xx^k) -\nabla f_j(\xx^\star) + \nabla f(\xx^\star) -\nabla f(\xx^k)}{\xx^k-\xx_k^i}\\
        &\quad + \inp{\nabla f_j(\xx^\star) -\nabla f(\xx^\star)}{\xx^k-\xx_k^i}\\
        &\overset{(iii)}{\leq} \norm{\nabla f_j(\xx_k^j) -\nabla f_j(\xx^k)}\norm{\xx_k^i-\xx^k} + \norm{\nabla f_j(\xx^k) -\nabla f_j(\xx^\star)}\norm{\xx_k^i-\xx^k} \\
        &\quad + \norm{\nabla f(\xx^\star) -\nabla f(\xx^k)}\norm{\xx_k^i-\xx^k} + \norm{\nabla f_j(\xx^\star) -\nabla f(\xx^\star)}\norm{\xx_k^i-\xx^k}\\
        &\overset{(iv)}{\leq} \frac{L}{2}\norm{\xx_k^j -\xx^k}^2 + \frac{3L}{2}\norm{\xx_k^i-\xx^k}^2 + \beta_{f_j}(\xx^\star,\xx^k) + \beta_{f}(\xx^\star,\xx^k) \\
        &\quad + \frac{a^q\norm{\nabla f_j(\xx^\star)-\nabla f(\xx^\star)}^q}{q} + \frac{\norm{\xx_k^i-\xx^k}^p}{pa^p},
    \end{align*}
    where in $(iii)$ we used the Cauchy-Schwarz inequality, and in $(iv)$ we used \Cref{assumption:convexity,assumption:smoothness} and Young's inequality (for any $a>0$). The value of $a$ will be chosen later.

    Now summing from $j=0$ to $i-1$ and from $i=0$ to $n-1$, we have:
    \begin{align*}
        &\sum_{i=0}^{n-1}\sum_{j=0}^{i-1}\inp{\gg_k^j - \nabla f(\xx^k)}{\xx^k-\xx_k^i}\\
        \leq & \frac{3(n-1)L}{2}\sum_{i=0}^{n-1} \norm{\xx_k^i-\xx^k}^2 + \frac{1}{pa^p}\sum_{i=0}^{n-1}i\norm{\xx_k^i-\xx^k}^p + \sum_{i=0}^{n-1}(n-1-i)\beta_{f_i}(\xx^\star,\xx^k) \\
        & + \frac{n(n-1)}{2}\beta_f(\xx^\star,\xx^k)+ \frac{a^q}{q}\sum_{i=0}^{n-1}(n-1-i)\norm{\nabla f_i(\xx^\star)-\nabla f(\xx^\star)}^q\\
        \overset{(v)}{\leq} & \frac{3nL}{2}\sum_{i=0}^{n-1} \norm{\xx_k^i-\xx^k}^2  + \frac{n}{pa^p}\sum_{i=0}^{n-1}\norm{\xx_k^i-\xx^k}^p + \frac{3n^2}{2}\beta_f(\xx^\star,\xx^k) + \frac{a^qn}{q}\sum_{i=0}^{n-1}\norm{\nabla f_i(\xx^\star)-\nabla f(\xx^\star)}^q\\
        \overset{(vi)}{\leq} & \frac{3nL}{2}\sum_{i=0}^{n-1} \norm{\xx_k^i-\xx^k}^2  + \frac{n}{pa^p}\sum_{i=0}^{n-1}\norm{\xx_k^i-\xx^k}^p + \frac{3n^2}{2}\beta_f(\xx^\star,\xx^k) + \frac{a^qn^2\sigma_\star^q}{q}\\
        \overset{(vii)}{\leq} & \frac{3nL}{2}\sum_{i=0}^{n-1} \norm{\xx_k^i-\xx^k}^2  + \frac{n}{pa^p}\sum_{i=0}^{n-1}\norm{\xx_k^i-\xx^k}^p + \frac{3n^2}{2}(F(\xx^k)-F^\star) + \frac{a^qn^2\sigma_\star^q}{q}
    \end{align*}
    where in $(v)$ we used the fact that $\sum_{i=0}^{n-1}\beta_{f_i}(\xx^\star,\xx^k)=n\beta_f(\xx^\star,\xx^k)$ 
    , and in $(vi)$ we used \Cref{assumption:heavy-tail}. In $(vii)$ we used \Cref{fact:bregman-composite}.

    Therefore,
    \begin{align*}
        & \frac{2\gamma -nL}{2}\sum_{i=0}^{n-1}\norm{\xx_k^i-\xx^k}^2 + \frac{\lambda}{p2^{p-1}}\sum_{i=0}^{n-1}\norm{\xx_k^i-\xx^k}^p \\
        \leq & \frac{n^2}{2} \big(F(\xx^k) - F^\star\big) + \sum_{i=0}^{n-1}\sum_{j=0}^{i-1}\inp{\gg_k^j - \nabla f(\xx^k)}{\xx^k-\xx_k^i} \\
        \leq & \frac{3nL}{2}\sum_{i=0}^{n-1} \norm{\xx_k^i-\xx^k}^2  + \frac{n}{pa^p}\sum_{i=0}^{n-1}\norm{\xx_k^i-\xx^k}^p + 2n^2(F(\xx^k)-F^\star) + \frac{a^qn^2\sigma_\star^q}{q}.
    \end{align*}
    Rearranging the above, we have:
    \[
        2(\gamma-2nL)\sum_{i=0}^{n-1}\norm{\xx_k^i-\xx^k}^2 + \frac{1}{p} \big(\frac{\lambda}{2^{p-1}}-\frac{n}{a^p}\big)\norm{\xx_k^i-\xx^k}^p \leq 2n^2(F(\xx^k)-F^\star) + \frac{a^qn^2\sigma_\star^q}{q}.
    \]
    Now we can choose $a$ such that $a^p=\frac{n2^{p-1}}{\lambda}$, and assuming that $\gamma \geq 4nL$, we then get the desired result.
\end{proof}

We can now put everything together. In particular, plugging \Cref{lem:drift-bound} and \Cref{lem:error-bound} into \Cref{lem:epoch-descent}, and assuming that $\gamma \geq 4nL$, we have:
\begin{equation}
    \label{eq:final-descent}
    \begin{aligned}
        & n \sum_{k=0}^{K-1}\big(F(\xx^{k+1})-F^\star\big) + \gamma\sum_{k=0}^{K-1}\norm{\xx^{k+1}-\xx^k}^2 + \frac{4\lambda}{p2^{p-2}}\sum_{k=0}^{K-1}\norm{\xx^{k+1}-\xx^k}^p \\
        \leq & 2\gamma R_0^2 + \frac{4\lambda R_0^p}{p} - \frac{4\gamma R_K^2}{2} - \frac{4\lambda R_K^p}{p2^{p-2}} + \frac{12n^2L(F(\xx_0)-F^\star)}{\gamma} + \frac{12Kn^{q+1}L\sigma_\star^q}{q\gamma\lambda^{q-1}}.
    \end{aligned}
\end{equation}
To tune the stepsizes $\gamma$ and $\lambda$ jointly, we apply \Cref{lem:stepsize-gamma-lambda} and obtain the following main convergence result:
\begin{restatable}{theorem}{MainConvergence}
    \label{thm:main-convergence}
    Given \Cref{assumption:convexity,assumption:smoothness,assumption:heavy-tail}, by choosing:
    \[
        \lambda = \left(\frac{3pn^{q+1}L\sigma^q_{\star}K}{q\gamma R_0^q}\right)^{\frac{1}{q}},\quad \gamma = \max\left\{4nL, \sqrt{\frac{12n^2L(F(\xx_0)-F^\star)}{2R_0^2}}, \left(\frac{12\cdot 2^{q-2}Kn^{q+1}L\sigma_\star^q}{p^{q-1}R_0^q}\right)^{\frac{1}{q+1}}\right\}
    \]
    we have that it takes at most
    \begin{equation}
        \label{eq:main-convergence}
        K = \frac{24LR_0^2 + 48\sqrt{LR_0^2(F(\xx_0)-F^\star)}}{\varepsilon} + \frac{4^{\frac{2q+3}{q}}R_0^{\frac{q+2}{q}}L^{\frac{1}{q}}\sigma_\star}{p^{\frac{q-1}{q}}q^{\frac{1}{q}}\varepsilon^{\frac{q+1}{q}}}
    \end{equation}
    epochs of \Cref{alg:da-power-prox-incremental} such that $\frac{1}{K}\sum_{k=0}^{K-1}\big(F(\xx^{k+1})-F^\star\big) \leq \varepsilon$.
\end{restatable}
\begin{remark}
    \label{remark:main-convergence}
    We first comment on the $\frac{48\sqrt{LR_0^2(F(\xx_0)-F^\star)}}{\varepsilon}$ term. This term is a common artifact in the analysis of composite optimization problems~\citep{gao2026composite}. When $\psi \equiv 0$, by \Cref{assumption:smoothness} we have $F(\xx_0)-F^\star \leq \frac{LR_0^2}{2}$, and the $\frac{48\sqrt{LR_0^2(F(\xx_0)-F^\star)}}{\varepsilon}$ term can be absorbed into the $\frac{24LR_0^2}{\varepsilon}$ term. When $\psi$ is non-trivial, this term can be absorbed via an additional (stochastic) gradient step at the start of the algorithm which reduces the initial suboptimality to the order of $LR_0^2$ plus some noise terms. In this work we omit this additional consideration as it does not affect the asymptotic rates, and we refer interested readers to \citet{gao2026composite} for more details.
\end{remark}

We note again that the above convergence rate concerns the \emph{number of epochs}. The total number of iterations (equivalently, gradient oracle calls) is $T=Kn$. In particular, the asymptotically dominant term for the number of iterations is $ \frac{4^{\frac{2q+3}{q}}nR_0^{\frac{q+2}{q}}L^{\frac{1}{q}}\sigma_\star}{p^{\frac{q-1}{q}}q^{\frac{1}{q}}\varepsilon^{\frac{q+1}{q}}}$. Comparing this with the $ \frac{2^{\frac{2-p}{q-1}}\sigma_\star^{\frac{q}{q-1}}R_0^{\frac{q}{q-1}}}{p\varepsilon^{\frac{q}{q-1}}}$ asymptotic rate of \algname{SGD}, we note that $\nicefrac{1}{\varepsilon^{\frac{q+1}{q}}}$ is a strictly better dependence on $\varepsilon$ than $\nicefrac{1}{\varepsilon^{\frac{q}{q-1}}}$ for any $q>1$.

\section{Experiments}
\label{sec:experiments}
\begin{wrapfigure}{r}{0.5\textwidth}           
      \centering
      \includegraphics[width=0.48\textwidth]{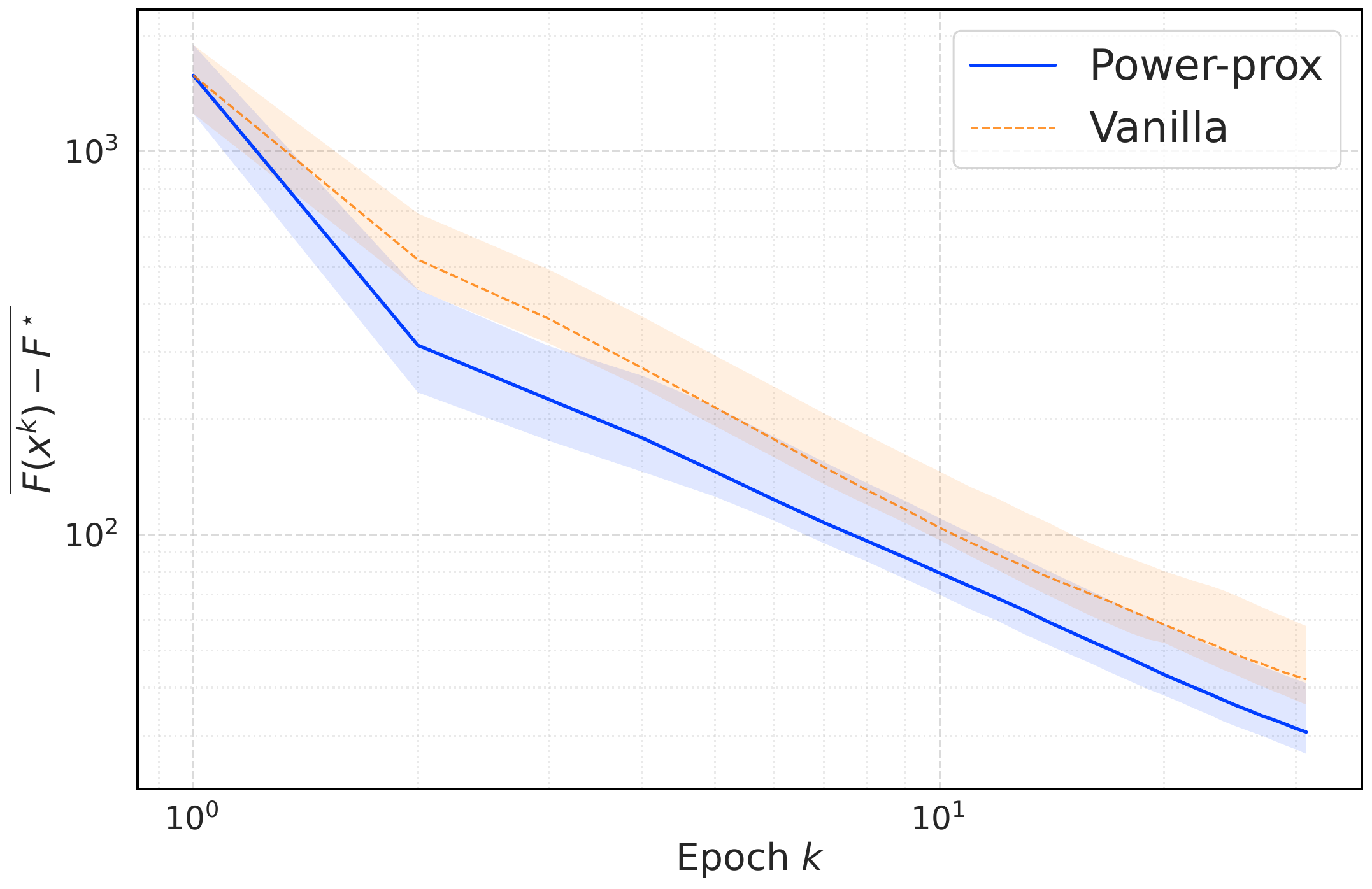}
      \caption{Synthetic least squares with heavy-tailed Pareto noise
  ($\alpha=1.5$,                          
      $n=500$, $d=50$).                                                         
      The overline denotes the running average over epochs $j=1,\dots,k$. Median
   and                                                                          
      IQR over $20$ seeds, per-seed best hyperparameters. Our method            
  consistently                                     
      outperforms the vanilla incremental gradient method.}
      \label{fig:synth_gap_avg}                                   \vspace{-0.2cm}          
  \end{wrapfigure}
In this section we present some preliminary numerical evaluations of our algorithm on a synthetic least-squares problem. We provide anonymized code for our experiments \href{https://anonymous.4open.science/r/synth-ls-htail-1698/README.md}{here}.  
Following \citet{hubler2025gradient}, we model      
  heavy-tailed gradient noise via the Pareto distribution. We construct a
  synthetic least-squares problem                                           
  $F(x) = \tfrac{1}{2}\sum_{i=1}^{n} (a_i^\top x - b_i)^2$
  from a standard Gaussian, row-normalized to unit norm, with each component
   $1$-smooth ($L = 1$). 
The target $x_{\text{target}}$ is drawn uniformly on   
  the unit sphere, and labels are generated as $b_i = a_i^\top x_{\text{target}}
   + \xi_i$, where $\xi_i = \sigma_i U_i^{-1/\alpha}$ is symmetric Pareto noise 
  with tail index $\alpha = 1.5$, $U_i \sim \text{Uniform}(0,1)$, and     
  random sign $\sigma_i \in {\pm 1}$. With $\alpha < 2$ the noise has infinite
  variance, but the $q$-th absolute moment is finite for any $q < \alpha$; we 
  take $q = 1.3$ throughout (so $p \approx 4.33$). To make the noise
  scale comparable across seeds, we rescale $\xi$ by a closed-form constant so
  that the empirical residual moment $\hat{\sigma}_\star^q := \tfrac{1}{n}\sum_i 
  |b_i - a_i^\top x^\star|^q$ matches the population value $\alpha/(\alpha - q)$ 
  exactly;
  
We compare \Cref{alg:da-power-prox-incremental} against vanilla dual averaging incremental gradient (i.e. $p=2$) over a run of $K=30$ epochs and $n=500$ samples. We repeat the experiment with $20$ random seeds.  We grid-search the optimal hyperparameters $\lambda, \gamma$ in the range $\{a \cdot 10^b : a \in \{1, 3, 5, 7\},\ b \in \{-3, -2, -1, 0, 1\}\}$ for each random seed and each method. We then report the median and interquartile range of these per-seed best curves across the $20$ seeds. The results are summarized in \Cref{fig:synth_gap_avg}. We see that \Cref{alg:da-power-prox-incremental} performs consistently better than the vanilla incremental gradient for all epochs and seeds, suggesting that our method outperforms the vanilla method in the heavy-tailed noise regime.

\section{Conclusion, limitation, and outlook}
\label{sec:conclusion}
In this paper, we investigated the theory of incremental methods under heavy-tailed noise, a setting that is underexplored yet central in modern ML training practice. We proposed a novel dual averaging power-prox method for incremental gradients and, to the best of our knowledge, provided the first rigorous convergence analysis for it under heavy-tailed noise. We showed that the power-prox method, a general technique that can address heavy-tailed noise, can effectively handle incremental updates. We also showed that dual averaging resolves a key technical challenge in the literature for composite objectives, where many existing works require the proximal operator to be applied only at the end of each epoch, risking undefined behavior during the epoch. Our analysis shows that even in the presence of heavy-tailed noise, our method can still achieve better asymptotic convergence rates than the corresponding \algname{SGD} method with i.i.d. sampling.

We believe that our work opens the door to a more comprehensive understanding of incremental gradient methods under heavy-tailed noise, and there are a few interesting directions for future work. First, it would be interesting to improve the asymptotically dominated term in the convergence rate. Like most existing works on incremental methods~\citep{mishchenko2022proximal,mishchenko2020random,liu2024last,josz2024proximal}, our analysis requires the stepsize $\gamma$ to depend on the epoch size $n$. This causes the optimization term in the convergence rate to depend on $n$ as well, which is not ideal. It appears to be a technical artifact of the epoch-wise nature of the analysis. It would be interesting to see whether this can be improved even in the bounded-variance setting. It would also be interesting to investigate whether we can obtain last-iterate convergence for our method. Our current analysis shows that an averaged iterate converges, but does not guarantee convergence of the last iterate. Finally, one can also consider extending our method to adjacent or more general settings, e.g., the non-convex case, generalized smoothness, or the distributed setting. We hope that our work can serve as a stepping stone for future research in this direction.

\section*{Acknowledgement}
The authors would like to thank Anton Rodomanov for explaining the power-prox framework and the helpful discussions.

\bibliographystyle{plainnat}
\bibliography{reference}

\appendix

\section{Auxiliary facts and technical lemmas}
\label{sec:auxiliary}
We first present a simple fact regarding the uniform convexity of the power function, which is a key technical tool in our analysis.

\begin{fact}
    \label{fact:uniform-convexity}
    The function $\frac{1}{p}\norm{\xx}^p$ where $p\geq 2$ is uniformly convex of order $p$ with constant $2^{2-p}$. 
    We say that a function $\phi$ is uniformly convex of order $p$ with constant $\mu$ if for any $\xx,\yy \in \R^d$, we have
    \begin{equation}
        \phi(\yy) \geq \phi(\xx) + \inp{\nabla \phi(\xx)}{\yy-\xx} + \frac{\mu}{p}\norm{\yy-\xx}^p.
    \end{equation}
\end{fact}

Next we present a simple fact regarding the Bregman divergence and the composite objective. We refer the readers to \citet{liu2024last} for a short proof of this fact.

\begin{fact}
    \label{fact:bregman-composite}
    Given $\xx^\star \eqdef \argmin_{\xx} F(\xx)$, we have for any $\xx\in \dom\psi$,
    \begin{equation}
        \label{eq:bregman-composite}   
        F(\xx) - F(\xx^\star) \geq \beta_f(\xx^\star, \xx) 
    \end{equation}
\end{fact}

We also present the following lemma for stepsize tuning.

\begin{lemma}
    \label{lem:stepsize-gamma-lambda}
    Suppose when $\gamma \geq E$, we have
    \[
        \frac{1}{K}\sum_{k=0}^{K-1}\alpha_k \leq \frac{\gamma A}{K} + \frac{\lambda B}{K} + \frac{C}{\gamma \lambda^{q-1}} + \frac{D}{\gamma K}.
    \]
    Then choosing $\lambda \eqdef  \left(\frac{CK}{\gamma B}\right)^{\frac{1}{q}}$ and $\gamma \eqdef \max\{E, \sqrt{D/A}, \frac{2^{\frac{q}{q+1}}K^{\frac{1}{q+1}}B^{\frac{q-1}{q+1}}C^{\frac{1}{q+1}}}{A^{\frac{q}{q+1}}}\}$, we have:
    \begin{equation}
        \label{eq:stepsize-gamma-lambda}
        \frac{1}{K}\sum_{k=0}^{K-1} \alpha_k \leq \frac{AE}{K} + \frac{2\sqrt{AD}}{K} + \left(\frac{2^{2q+1}AB^{q-1}C}{K^q}\right)^{\frac{1}{q+1}}
    \end{equation}
\end{lemma}
\begin{proof}

    For any $\gamma$, choose $\lambda$ such that 
    \[
        \frac{\lambda B}{K} = \frac{C}{\gamma\lambda^{q-1}},
    \]
    in other words, $\lambda = \left(\frac{CK}{\gamma B}\right)^{\frac{1}{q}}$. Then we have:
    \[
        \frac{1}{K}\sum_{k=0}^{K-1} \alpha_k \leq \frac{\gamma A}{K} + \frac{D}{\gamma K} + 2\left(\frac{B}{K}\right)^{\frac{q-1}{q}}\left(\frac{C}{\gamma}\right)^{\frac{1}{q}}
    \]
    Now choosing $\gamma =\max\{E, \sqrt{D/A}, \frac{2^{\frac{q}{q+1}}K^{\frac{1}{q+1}}B^{\frac{q-1}{q+1}}C^{\frac{1}{q+1}}}{A^{\frac{q}{q+1}}}\}$, we have:
    \[
        \frac{1}{K}\sum_{k=0}^{K-1} \alpha_k \leq \frac{AE}{K} + \frac{2\sqrt{AD}}{K} + \left(\frac{2^{2q+1}AB^{q-1}C}{K^q}\right)^{\frac{1}{q+1}}
    \]

\end{proof}

\section{Implementation of the power-prox}
\label{sec:implementation}
Here we briefly discuss the implementation of the power-prox method. In essence, at each iteration, for some point $\xx$ and some dual point $\gg$, we need to solve the following optimization problem:
\begin{equation}
    \label{eq:subproblem-general}
    \bar \xx \eqdef \argmin_{\yy\in \dom\psi} \bigg\{\inp{\gg}{\yy} + \psi(\yy) + \frac{\gamma}{2}\norm{\yy-\xx}^2 + \frac{\lambda}{p}\norm{\yy-\xx}^p\bigg\}.
\end{equation}
Note that one needs to scale the parameters and vectors properly to match the above formulation to \Cref{eq:subproblem-specific}. Solving \Cref{eq:subproblem-general} requires specific knowledge of the composite part $\psi$, and there is no generic solution to it. We work with the assumption that $\psi$ plus a polynomial should be easy to solve. For completeness, we discuss a specific procedure for when $\psi\equiv 0$. We note that the procedure is folklore and that power-prox has been used in the literature in different contexts; see, e.g., \citet{lu2018relatively}. 

When $\psi\equiv 0$, the problem reduces to
\[
    \bar \xx = \argmin_{\yy} \bigg\{\inp{\gg}{\yy} + \frac{\gamma}{2}\norm{\yy-\xx}^2 + \frac{\lambda}{p}\norm{\yy-\xx}^p\bigg\}.
\]      
By the first order optimality condition, we have:
\[
    \gg + \gamma(\bar \xx - \xx) + \lambda \norm{\bar \xx - \xx}^{p-2}(\bar \xx - \xx) = \0.
\]
For simplicity, denote $\dd = \bar \xx - \xx$. Then rearranging, we have:
\[
    \dd = -\frac{\gg}{\gamma + \lambda \norm{\dd}^{p-2}}.
\]
In particular, $\dd$ is in the opposite direction of $\gg$, and it suffices to solve for the magnitude of $\dd$. Let $r = \norm{\dd}$, then we have:
\[
    r = \frac{\norm{\gg}}{\gamma + \lambda r^{p-2}}.
\]
For $r>0$, the left-hand side is increasing while the right-hand side is decreasing. At $r=0$ the left-hand side is smaller than the right-hand side, and at $r\to +\infty$ the left-hand side is larger than the right-hand side. Therefore there exists a unique solution to the above equation. Moreover, this is a scalar polynomial equation of order $p-1$. For $p=3,4,5$, this can be solved in closed form. For other values of $p$, this can be solved efficiently using simple numerical procedures, e.g., the bisection method or Newton's method.

\section{Missing proofs from \Cref{sec:warmup}}
\label{sec:missing-proofs-sgd} 
In this section we present the missing proofs in \Cref{sec:warmup}. For simplicity, we define
\[
    \Psi_t(\xx ) \eqdef \sum_{s=0}^{t}\left(f(\xx_s) + \inp{\gg_s}{\xx-\xx_s}+\psi(\xx)\right) + \frac{\gamma}{2}\norm{\xx-\xx_0}^2 + \frac{\lambda}{p}\norm{\xx-\xx_0}^p.
\]
We also write $\Psi_t^\star\eqdef \Psi_t(\xx_{t+1})$. Recall that we define $\xx_{t+1}$ to be the minimizer of $\Psi_t$. It suffices to define $\Psi_{-1}\equiv0$.

By \Cref{fact:uniform-convexity}, we have that for any $\xx\in\dom\psi$:
\[
    \Psi_t(\xx) \geq \Psi_t^\star + \frac{\gamma}{2}\norm{\xx-\xx_{t+1}}^2 + \frac{\lambda}{p2^{p-2}} \norm{\xx-\xx_{t+1}}^p.
\]
Moreover, by the definition of $\Psi_t$, we have:
\begin{align*}
    \Psi_t(\xx) &= \sum_{s=0}^{t} \big(f(\xx_s) + \inp{\gg_s}{\xx-\xx_s}+\psi(\xx)\big) + \frac{\gamma}{2}\norm{\xx-\xx_0}^2 + \frac{\lambda}{p}\norm{\xx-\xx_0}^p\\
    & = \sum_{s=0}^{t} \big(f(\xx_s) + \inp{\nabla f(\xx_s)}{\xx-\xx_s}+\psi(\xx)\big) + \frac{\gamma}{2}\norm{\xx-\xx_0}^2 + \frac{\lambda}{p}\norm{\xx-\xx_0}^p\\
    &\quad + \sum_{s=0}^{t} \inp{\gg_s - \nabla f(\xx_s)}{\xx-\xx_s}\\
    & = \sum_{s=0}^{t} (F(\xx)- \beta_f(\xx_s,\xx))+ \frac{\gamma}{2}\norm{\xx-\xx_0}^2 + \frac{\lambda}{p}\norm{\xx-\xx_0}^p + \sum_{s=0}^{t} \inp{\gg_s - \nabla f(\xx_s)}{\xx-\xx_s}.
\end{align*}

Now by the definition of $\Psi_t^\star$, we have:
\begin{align*}
    \Psi_t^\star & = \Psi_{t-1}(\xx_{t+1}) + f(\xx_t) + \inp{\gg_t}{\xx_{t+1}-\xx_t} + \psi(\xx_{t+1}) \\
    & \overset{(i)}{\geq} \Psi_{t-1}^\star + \frac{\gamma}{2}\norm{\xx_{t+1}-\xx_t}^2 + \frac{\lambda}{p2^{p-2}} \norm{\xx_{t+1}-\xx_t}^p + f(\xx_t) + \inp{\gg_t}{\xx_{t+1}-\xx_t} + \psi(\xx_{t+1})\\
    &= \Psi_{t-1}^\star + \frac{\gamma}{2}\norm{\xx_{t+1}-\xx_t}^2 + \frac{\lambda}{p2^{p-2}} \norm{\xx_{t+1}-\xx_t}^p + f(\xx_t) + \inp{\nabla f(\xx_t)}{\xx_{t+1}-\xx_t} + \psi(\xx_{t+1})\\
    &\quad + \inp{\gg_t - \nabla f(\xx_t)}{\xx_{t+1}-\xx_t}\\
    & \overset{(ii)}{\geq} \Psi_{t-1}^\star + \frac{\gamma-L}{2}\norm{\xx_{t+1}-\xx_t}^2 + \frac{\lambda}{p2^{p-2}} \norm{\xx_{t+1}-\xx_t}^p + F(\xx_{t+1})  + \inp{\gg_t - \nabla f(\xx_t)}{\xx_{t+1}-\xx_t},
\end{align*}
where in $(i)$ we used \Cref{fact:uniform-convexity} and in $(ii)$ we used \Cref{assumption:smoothness}. Rearranging the terms and summing from $t=0$ to $T-1$, we have:
\begin{align*}
    &\sum_{t=0}^{T-1} F(\xx_{t+1}) + \frac{\gamma}{2}\sum_{t=0}^{T-1}\norm{\xx_{t+1}-\xx_t}^2 + \frac{\lambda}{p2^{p-2}} \sum_{t=0}^{T-1}\norm{\xx_{t+1}-\xx_t}^p\\
    \leq & \Psi_{T-1}^\star+ \sum_{t=0}^{T-1}\inp{\gg_t-\nabla f(\xx_t)}{\xx_t-\xx_{t+1}}\\
    \leq & \Psi_{T-1}(\xx) - \frac{\gamma}{2}\norm{\xx-\xx_T}^2 - \frac{\lambda}{p2^{p-2}} \norm{\xx-\xx_T}^p + \sum_{t=0}^{T-1}\inp{\gg_t-\nabla f(\xx_t)}{\xx_t-\xx_{t+1}}\\
    = & \sum_{t=0}^{T-1} (F(\xx)- \beta_f(\xx_t,\xx)) + \frac{\gamma}{2}\norm{\xx-\xx_0}^2 + \frac{\lambda}{p}\norm{\xx-\xx_0}^p - \frac{\gamma}{2}\norm{\xx-\xx_T}^2 - \frac{\lambda}{p2^{p-2}} \norm{\xx-\xx_T}^p \\
    & + \sum_{t=0}^{T-1}\inp{\gg_t-\nabla f(\xx_t)}{\xx-\xx_{t+1}}.
\end{align*}
Now rearranging again, taking expectations, and setting $\xx = \xx^\star$, we have:
\begin{align*}
    & \sum_{t=0}^{T-1}\Eb{F(\xx_{t+1})-F^\star} + \frac{\gamma}{2}\sum_{t=0}^{T-1}\Eb{\norm{\xx_{t+1}-\xx_t}^2} + \frac{\lambda}{p2^{p-2}} \sum_{t=0}^{T-1}\Eb{\norm{\xx_{t+1}-\xx_t}^p}\\
    \overset{(iii)}{\leq} & \frac{\gamma R_0^2}{2} + \frac{\lambda R_0^p}{p} + \sum_{t=0}^{T-1}\Eb{\inp{\gg_t-\nabla f(\xx_t)}{\xx^\star-\xx_{t+1}}},
\end{align*}
where in $(iii)$ we used \Cref{assumption:convexity}, which implies that $\beta_f(\xx_t,\xx) \geq 0$. Now, since $\gg_t\eqdef \nabla f_{i_t}(\xx_t)$ where $i_t$ is sampled uniformly at random from $\{0,\ldots,n-1\}$ and independently from the past (in particular, independently from $\xx_t$ and $\xx^\star$), we have that $\Eb{\inp{\gg_t-\nabla f(\xx_t)}{\xx^\star}} = \Eb{\inp{\gg_t-\nabla f(\xx_t)}{\xx_t}}=0$.
Therefore:
\begin{align*}
    & \sum_{t=0}^{T-1}\Eb{F(\xx_{t+1})-F^\star} + \frac{\gamma}{2}\sum_{t=0}^{T-1}\Eb{\norm{\xx_{t+1}-\xx_t}^2} + \frac{\lambda}{p2^{p-2}} \sum_{t=0}^{T-1}\Eb{\norm{\xx_{t+1}-\xx_t}^p}\\
    \leq & \frac{\gamma R_0^2}{2} + \frac{\lambda R_0^p}{p} + \sum_{t=0}^{T-1}\Eb{\inp{\gg_t-\nabla f(\xx_t)}{\xx_t-\xx_{t+1}}}.
\end{align*}
To upper bound the residual error, we use Young's inequality:
\begin{align*}
    &\Eb{\inp{\gg_t-\nabla f(\xx_t)}{\xx_t-\xx_{t+1}}} \\
    = & \Eb{\inp{\nabla f_{i_t}(\xx_t)-\nabla f_{i_t}(\xx^\star) +\nabla f(\xx^\star)-\nabla f(\xx_t)}{\xx_t-\xx_{t+1}}} + \Eb{\inp{\nabla f_{i_t}(\xx^\star)-\nabla f(\xx^\star)}{\xx_t-\xx_{t+1}}}\\
    \overset{(iv)}{\leq} & \frac{1}{2\gamma}\Eb{\norm{\nabla f_{i_t}(\xx_t)-\nabla f_{i_t}(\xx^\star) +\nabla f(\xx^\star)-\nabla f(\xx_t)}^2} + \frac{\gamma}{2}\Eb{\norm{\xx_t-\xx_{t+1}}^2} + \frac{1}{q\alpha^q}\Eb{\norm{\nabla f_{i_t}(\xx^\star)-\nabla f(\xx^\star)}^q}\\
    & + \frac{\alpha^p}{p}\Eb{\norm{\xx_t-\xx_{t+1}}^p}\\
    \leq & \frac{1}{\gamma} \Eb{\norm{\nabla f_{i_t}(\xx_t)-\nabla f_{i_t}(\xx^\star)}^2+\norm{\nabla f(\xx^\star)-\nabla f(\xx_t)}^2} +\frac{\gamma}{2}\Eb{\norm{\xx_t-\xx_{t+1}}^2} + \frac{1}{q\alpha^q}\Eb{\norm{\nabla f_{i_t}(\xx^\star)-\nabla f(\xx^\star)}^q}\\
    & + \frac{\alpha^p}{p}\Eb{\norm{\xx_t-\xx_{t+1}}^p}\\
    \overset{(v)}{\leq} & \frac{2L}{\gamma}\Eb{\beta_{f_{i_t}}(\xx^\star,\xx_t)+\beta_f(\xx^\star,\xx_t)} +\frac{\gamma}{2}\Eb{\norm{\xx_t-\xx_{t+1}}^2} + \frac{1}{q\alpha^q}\Eb{\norm{\nabla f_{i_t}(\xx^\star)-\nabla f(\xx^\star)}^q} + \frac{\alpha^p}{p}\Eb{\norm{\xx_t-\xx_{t+1}}^p}\\
    \overset{(vi)}{\leq} & \frac{4L}{\gamma}\Eb{\beta_f(\xx^\star,\xx_t)} +\frac{\gamma}{2}\Eb{\norm{\xx_t-\xx_{t+1}}^2} + \frac{1}{q\alpha^q}\Eb{\norm{\nabla f_{i_t}(\xx^\star)-\nabla f(\xx^\star)}^q} + \frac{\alpha^p}{p}\Eb{\norm{\xx_t-\xx_{t+1}}^p}\\
    \overset{(vii)}{\leq} & \frac{4L}{\gamma}\Eb{F(\xx_t)-F^\star} +\frac{\gamma}{2}\Eb{\norm{\xx_t-\xx_{t+1}}^2} + \frac{1}{q\alpha^q}\Eb{\norm{\nabla f_{i_t}(\xx^\star)-\nabla f(\xx^\star)}^q} + \frac{\alpha^p}{p}\Eb{\norm{\xx_t-\xx_{t+1}}^p}\\
    \overset{(viii)}{\leq} & \frac{4L}{\gamma}\Eb{F(\xx_t)-F^\star} +\frac{\gamma}{2}\Eb{\norm{\xx_t-\xx_{t+1}}^2} + \frac{1}{q\alpha^q}\sigma_\star^q + \frac{\alpha^p}{p}\Eb{\norm{\xx_t-\xx_{t+1}}^p},
\end{align*}
where in $(iv)$ we used Young's inequality, in $(v)$ we used \Cref{assumption:smoothness}, in $(vi)$ we used the definition of $i_t$, which implies that $\Eb{\beta_{f_{i_t}}(\xx^\star,\xx_t)}=\Eb{\beta_f(\xx^\star,\xx_t)}$, in $(vii)$ we used \Cref{fact:bregman-composite}, and in $(viii)$ we used \Cref{assumption:heavy-tail}.

Now plugging the above bound back into the previous inequality, we have:
\[
    \sum_{t=0}^{T-1}\Eb{F(\xx_{t+1})-F^\star}  +\frac{1}{p} \left(\frac{\lambda}{2^{p-2}}-\alpha^p \right)\sum_{t=0}^{T-1}\Eb{\norm{\xx_{t+1}-\xx_t}^p}
    \leq  \frac{\gamma R_0^2}{2} + \frac{\lambda R_0^p}{p} + \frac{4L}{\gamma}\sum_{t=0}^{T-1}\Eb{F(\xx_t)-F^\star} + \frac{\sigma_\star^qT}{q\alpha^q} 
\]
Assuming that $\gamma\geq 8L$, setting $\alpha^p = \frac{\lambda}{2^{p-2}}$, and rearranging the terms, and dividing by $T$ on both sides, we have:
\[
    \frac{1}{T}\sum_{t=0}^{T-1}\Eb{F(\xx_{t+1})-F^\star}  \leq  \frac{\gamma R_0^2}{T} + \frac{2\lambda R_0^p}{pT} + \frac{8L}{\gamma T }\left(F(\xx_0)-F^\star\right) + \frac{2^{\frac{p-2}{p-1}}\sigma^q_\star }{q\lambda^{q-1}}.
\]
Now we set $\gamma = \max\{8L, \sqrt{\frac{8L(F(\xx_0)-F^\star)}{R_0^2}}\}$, and $\lambda = \left(\frac{2p\sigma_\star^qT}{qR_0^p}\right)^{\frac{1}{q}}$, then we have:
\[
    \frac{1}{T}\sum_{t=0}^{T-1}\Eb{F(\xx_{t+1})-F^\star}  \leq  \frac{8LR_0^2 + \sqrt{8L(F(\xx_0)-F^\star)R_0^2}}{T} + \frac{8\sigma_\star R_0}{T^{\frac{q-1}{q}}}
\]
Setting each term less than $\frac{\varepsilon}{2}$, we get the desired result.

\section{Missing proofs from \Cref{sec:sketch-convergence-analysis}}
\label{sec:missing-proofs-incremental} 
In this section we present the missing proofs from \Cref{sec:sketch-convergence-analysis}.

We now present the proof of \Cref{lem:epoch-descent}. Again, for ease of presentation, write 
\[
    \Phi_k(\xx) \eqdef \sum_{e=0}^{k}\sum_{i=0}^{n-1} \big(f(\xx^e)+ \inp{\gg_e^i}{\xx-\xx^e}+\psi(\xx)\big)  + \frac{\gamma}{2} \norm{\xx - \xx_0}^2  + \frac{\lambda}{p} \norm{\xx-\xx_0}^p.
\]
and $\Phi_k^\star\eqdef \Phi_k(\xx^{k+1})$. We also write $\Phi_{-1} \equiv 0$ so that $\xx_0=\xx^0 = \argmin_{\xx\in\dom\psi}\Phi_{-1}(\xx)$.

\EpochDescent*

\begin{proof}
    By the uniform convexity of $\Phi_k$ (of order both $2$ and $p$), we have for any $\xx\in \dom\psi$,
    \[
        \Phi_k(\xx) \geq \Phi_k^\star + \frac{\gamma}{2}\norm{\xx-\xx^{k+1}}^2 + \frac{\lambda}{p2^{p-2}}\norm{\xx-\xx^{k+1}}^p.
    \]
    Moreover,
    \begin{align*}
        \Phi_k(\xx) &= \sum_{e=0}^{k}\sum_{i=0}^{n-1} \big(f(\xx^e)+ \inp{\gg_e^i}{\xx-\xx^e}+\psi(\xx)\big)  + \frac{\gamma}{2} \norm{\xx - \xx_0}^2  + \frac{\lambda}{p} \norm{\xx-\xx_0}^p \\
        &= \sum_{e=0}^{k}\sum_{i=0}^{n-1} \big(f(\xx^e)+ \inp{\nabla f(\xx^e)}{\xx-\xx^e}+\psi(\xx)\big)  + \frac{\gamma}{2} \norm{\xx - \xx_0}^2  + \frac{\lambda}{p} \norm{\xx-\xx_0}^p \\
        &\quad + \sum_{e=0}^{k}\sum_{i=0}^{n-1} \inp{\gg_e^i - \nabla f(\xx^e)}{\xx-\xx^e} \\
        &= n\sum_{e=0}^{k}F(\xx) - n\sum_{e=0}^{k}\beta_f(\xx^e,\xx)+ \frac{\gamma}{2} \norm{\xx - \xx_0}^2  + \frac{\lambda}{p} \norm{\xx-\xx_0}^p+ \sum_{e=0}^{k}\sum_{i=0}^{n-1} \inp{\gg_e^i - \nabla f(\xx^e)}{\xx-\xx^e}.
    \end{align*}
    Now by the definition of $\Phi_k$, we have:
    \begin{align*}
        \Phi_k^\star &= \Phi_{k-1}(\xx^{k+1}) + \sum_{i=0}^{n-1} \big(f(\xx^k)+ \inp{\gg_k^i}{\xx^{k+1}-\xx^k}+\psi(\xx^{k+1})\big) \\
        &\overset{(i)}{\geq} \Phi_{k-1}^\star + \frac{\gamma}{2}\norm{\xx^{k+1}-\xx^k}^2 + \frac{\lambda}{p2^{p-2}}\norm{\xx^{k+1}-\xx^k}^p + \sum_{i=0}^{n-1} \big(f(\xx^k)+ \inp{\gg_k^i}{\xx^{k+1}-\xx^k}+\psi(\xx^{k+1})\big)\\
        &= \Phi_{k-1}^\star + \frac{\gamma}{2}\norm{\xx^{k+1}-\xx^k}^2 + \frac{\lambda}{p2^{p-2}}\norm{\xx^{k+1}-\xx^k}^p + \sum_{i=0}^{n-1} \big(f(\xx^k)+ \inp{\nabla f(\xx^k)}{\xx^{k+1}-\xx^k}+\psi(\xx^{k+1})\big)\\
        &\quad + \sum_{i=0}^{n-1} \inp{\gg_k^i - \nabla f(\xx^k)}{\xx^{k+1}-\xx^k} \\
        &\overset{(ii)}{\geq} \Phi_{k-1}^\star + \frac{\gamma-nL}{2}\norm{\xx^{k+1}-\xx^k}^2 + \frac{\lambda}{p2^{p-2}}\norm{\xx^{k+1}-\xx^k}^p + nF(\xx^{k+1}) + \sum_{i=0}^{n-1} \inp{\gg_k^i - \nabla f(\xx^k)}{\xx^{k+1}-\xx^k},
    \end{align*}
    where in $(i)$ we used the uniform convexity, and in $(ii)$ we used \Cref{assumption:smoothness}, the smoothness of $f$.

    Now rearranging, summing from $k=0$ to $K-1$ and applying the previous inequalities, we have for any $\xx\in \dom\psi$,
    \begin{align*}
        &n\sum_{k=0}^{K-1}F(\xx^{k+1}) + \frac{\gamma-nL}{2}\sum_{k=0}^{K-1}\norm{\xx^{k+1}-\xx^k}^2 + \frac{\lambda}{p2^{p-2}}\sum_{k=0}^{K-1}\norm{\xx^{k+1}-\xx^k}^p \\
        \leq & \Phi_K^\star -\Phi_{-1}^\star+ \sum_{k=0}^{K-1}\sum_{i=0}^{n-1} \inp{\gg_k^i - \nabla f(\xx^k)}{\xx^k-\xx^{k+1}}\\
        = &  \Phi_K^\star+ \sum_{k=0}^{K-1}\sum_{i=0}^{n-1} \inp{\gg_k^i - \nabla f(\xx^k)}{\xx^k-\xx^{k+1}}\\
        \leq & \Phi_K(\xx) - \frac{\gamma}{2}\norm{\xx-\xx^{K}}^2 - \frac{\lambda}{p2^{p-2}}\norm{\xx-\xx^{K}}^p + \sum_{k=0}^{K-1}\sum_{i=0}^{n-1} \inp{\gg_k^i - \nabla f(\xx^k)}{\xx^k-\xx^{k+1}}\\
        =  & n\sum_{k=0}^{K-1}F(\xx) + \frac{\gamma}{2} \norm{\xx - \xx_0}^2  + \frac{\lambda}{p} \norm{\xx-\xx_0}^p - \frac{\gamma}{2}\norm{\xx-\xx^{K}}^2 - \frac{\lambda}{p2^{p-2}}\norm{\xx-\xx^{K}}^p\\
        &  + \sum_{k=0}^{K-1}\sum_{i=0}^{n-1}\big( \inp{\gg_k^i - \nabla f(\xx^k)}{\xx-\xx^{k+1}}-\beta_f(\xx^k,\xx)\big).
    \end{align*}
    Rearranging and taking $\xx\eqdef \xx^\star$, we get \Cref{eq:epoch-descent}.
\end{proof}
Next we present the proof of \Cref{lem:error-bound}.
\ErrorBound*
\begin{proof}
    First consider the individual terms:
    \begin{align*}
        \inp{\gg_k^i - \nabla f(\xx^k)}{\xx^\star-\xx^{k+1}}-\beta_f(\xx^k,\xx^\star) &= \inp{\nabla f_i(\xx_k^i) - \nabla f(\xx^k)}{\xx^\star-\xx^{k+1}}-\beta_f(\xx^k,\xx^\star)\\
        &= \inp{\nabla f_i(\xx_k^i) - \nabla f(\xx^k)}{\xx^k-\xx^{k+1}}+ \inp{\nabla f_i(\xx_k^i) - \nabla f(\xx^k)}{\xx^\star-\xx^k}\\
        &\quad -\big(f(\xx^\star)-f(\xx^k)-\inp{\nabla f(\xx^k)}{\xx^\star-\xx^k}\big)\\
        &=\inp{\nabla f_i(\xx_k^i) - \nabla f(\xx^k)}{\xx^k-\xx^{k+1}}+ \inp{\nabla f_i(\xx_k^i) }{\xx^\star-\xx^k} \\
        &\quad - \big(f(\xx^\star)-f(\xx^k)\big)\\
        &=\inp{\nabla f_i(\xx_k^i) - \nabla f(\xx^k)}{\xx^k-\xx^{k+1}}+ \inp{\nabla f_i(\xx_k^i) }{\xx^\star-\xx_k^i} \\
        &\quad + \inp{\nabla f_i(\xx_k^i)}{\xx_k^i-\xx^k}  - \big(f(\xx^\star)-f(\xx^k)\big)\\
        &\overset{(i)}{\leq} \inp{\nabla f_i(\xx_k^i) - \nabla f(\xx^k)}{\xx^k-\xx^{k+1}}+ f_i(\xx^\star) - f_i(\xx_k^i) \\
        &\quad + \inp{\nabla f_i(\xx_k^i)-\nabla f_i(\xx^k)}{\xx_k^i-\xx^k} + \inp{\nabla f_i(\xx^k)}{\xx_k^i-\xx^k} - \big(f(\xx^\star)-f(\xx^k)\big)\\
        &\overset{(ii)}{\leq}\inp{\nabla f_i(\xx_k^i) - \nabla f(\xx^k)}{\xx^k-\xx^{k+1}}+ f_i(\xx^\star) - f_i(\xx^k) \\
        &\quad + \inp{\nabla f_i(\xx_k^i)-\nabla f_i(\xx^k)}{\xx_k^i-\xx^k}  - \big(f(\xx^\star)-f(\xx^k)\big),
    \end{align*}
    where in $(i)$ and $(ii)$ we used the convexity of $f_i$. Now summing over $i=0$ to $n-1$, we obtain:
    \begin{align*}
        & \sum_{i=0}^{n-1}\big( \inp{\gg_k^i - \nabla f(\xx^k)}{\xx^\star-\xx^{k+1}}-\beta_f(\xx^k,\xx^\star)\big) \\
        \overset{(iii)}{\leq} & \sum_{i=0}^{n-1}\big(\inp{\nabla f_i(\xx_k^i) - \nabla f(\xx^k)}{\xx^k-\xx^{k+1}} + \inp{\nabla f_i(\xx_k^i)-\nabla f_i(\xx^k)}{\xx_k^i-\xx^k} \big)\\
        \overset{(iv)}{\leq} & \sum_{i=0}^{n-1}\big(\inp{\nabla f_i(\xx_k^i) - \nabla f(\xx^k)}{\xx^k-\xx^{k+1}} + L\norm{\xx_k^i-\xx^k}^2 \big)\\
        = & \sum_{i=0}^{n-1}\big(\inp{\nabla f_i(\xx_k^i) - \nabla f_i(\xx^k)}{\xx^k-\xx^{k+1}} + \inp{\nabla f_i(\xx^k)-\nabla f(\xx^k)}{\xx^k-\xx^{k+1}} + L\norm{\xx_k^i-\xx^k}^2 \big)\\
        \overset{(v)}{=} & \sum_{i=0}^{n-1}\big(\inp{\nabla f_i(\xx_k^i) - \nabla f_i(\xx^k)}{\xx^k-\xx^{k+1}} + L\norm{\xx_k^i-\xx^k}^2 \big)\\
        \overset{(vi)}{\leq} & \frac{Ln}{2}\norm{\xx^{k+1}-\xx^k}^2 + \frac{3L}{2}\sum_{i=0}^{n-1}\norm{\xx_k^i-\xx^k}^2,
    \end{align*}
    where in $(iii)$ we used the fact that $\sum_{i=0}^{n-1} f_i(\xx^\star) - f_i(\xx^k) = n (f(\xx^\star) - f(\xx^k))$, in $(iv)$ we used \Cref{assumption:smoothness}. In $(v)$ we used the fact that $\sum_{i=0}^{n-1}\inp{\nabla f_i(\xx^k)-\nabla f(\xx^k)}{\xx^k-\xx^{k+1}} = 0$. Finally in $(vi)$ we used Young's inequality and \Cref{assumption:smoothness}.
\end{proof}

Now we give the full proof of \Cref{lem:drift-bound}.
\DriftBound*
\begin{proof}
    For simplicity, for each epoch $k$, write $\Phi_k(\xx) \eqdef \sum_{e=0}^{k}\sum_{i=0}^{n-1} \big(f(\xx^e)+ \inp{\gg_e^i}{\xx-\xx^e}+\psi(\xx)\big)  + \frac{\gamma}{2} \norm{\xx - \xx_0}^2  + \frac{\lambda}{p} \norm{\xx-\xx_0}^p$, and write $\Phi_k^\star \eqdef \Phi_k(\xx^{k+1})$. We write $\Phi_{-1}\equiv 0$.
    
    By the definition of $\Phi_k^i$, we have:
    \begin{align*}
        \Phi_k^{i-1}(\xx_k^i) & = \Phi_{k-1}(\xx_k^i) + \sum_{j=0}^{i-1} \big(f(\xx^k)+ \inp{\gg_k^j}{\xx_k^i-\xx^k}+\psi(\xx_k^i)\big) \\
        &\overset{(i)}{\geq} \Phi_{k-1}^\star + \frac{\gamma}{2}\norm{\xx_k^i-\xx^k}^2 + \frac{\lambda}{p2^{p-2}}\norm{\xx_k^i-\xx^k}^p + \sum_{j=0}^{i-1} \big(f(\xx^k)+ \inp{\gg_k^j}{\xx_k^i-\xx^k}+\psi(\xx_k^i)\big) \\
        &= \Phi_{k-1}^\star + \frac{\gamma}{2}\norm{\xx_k^i-\xx^k}^2 + \frac{\lambda}{p2^{p-2}}\norm{\xx_k^i-\xx^k}^p + \sum_{j=0}^{i-1} \big(f(\xx^k)+ \inp{\nabla f(\xx^k)}{\xx_k^i-\xx^k}+\psi(\xx_k^i)\big) \\
        &\quad + \sum_{j=0}^{i-1} \inp{\gg_k^j - \nabla f(\xx^k)}{\xx_k^i-\xx^k} \\
        &\overset{(ii)}{\geq} \Phi_{k-1}^\star + \frac{\gamma-iL}{2}\norm{\xx_k^i-\xx^k}^2 + \frac{\lambda}{p2^{p-2}}\norm{\xx_k^i-\xx^k}^p + iF(\xx_k^i)  + \sum_{j=0}^{i-1} \inp{\gg_k^j - \nabla f(\xx^k)}{\xx_k^i-\xx^k},
    \end{align*}
    where in $(i)$ we used \Cref{fact:uniform-convexity} and in $(ii)$ we used \Cref{assumption:smoothness}.

    Moreover, again by \Cref{fact:uniform-convexity} and the definitions of $\Phi_k^i$ and $\Phi_{k-1}$, we have:
    \begin{align*}
        &\Phi_k^{i-1}(\xx_k^i) + \frac{\gamma}{2}\norm{\xx_k^i-\xx^k}^2 + \frac{\lambda}{p2^{p-2}} \norm{\xx_k^i-\xx^k}^p\\
        \leq & \Phi_k^{i-1}(\xx^k) = \Phi_{k-1}^\star + iF(\xx^k) 
    \end{align*}
    Putting the above two inequalities together, we have:
    \begin{align*}
        \frac{2\gamma -iL}{2} \norm{\xx_k^i-\xx^k}^2 + \frac{\lambda}{p2^{p-1}}\norm{\xx_k^i-\xx^k}^p &\leq i\big(F(\xx^k)-F(\xx_k^i)\big) +  \sum_{j=0}^{i-1} \inp{\gg_k^j - \nabla f(\xx^k)}{\xx^k-\xx_k^i}\\
        &\leq i\big(F(\xx^k)-F^\star\big) +  \sum_{j=0}^{i-1} \inp{\gg_k^j - \nabla f(\xx^k)}{\xx^k-\xx_k^i}
    \end{align*}
    Now we move on to upper bound $\sum_{j=0}^{i-1} \inp{\gg_k^j - \nabla f(\xx^k)}{\xx^k-\xx_k^i}$. This is where it becomes clear why the power-prox term is crucial. It is clear that the gradient error $\gg_k^j - \nabla f(\xx^k)$, after some possible manipulations, has to be bounded via some Cauchy-Schwarz type inequalities. Since the noise is heavy-tailed with power $q$, Cauchy-Schwarz will raise the drifts $\norm{\xx_k^i-\xx^k}$ to the conjugate power $p$. This will then be nicely canceled with the power-prox induced terms. We have:
    \begin{align*}
        \inp{\gg_k^j - \nabla f(\xx^k)}{\xx^k-\xx_k^i} &= \inp{\nabla f_j(\xx_k^j) -\nabla f_j(\xx^k) +\nabla f_j(\xx^k) -\nabla f_j(\xx^\star) + \nabla f(\xx^\star) -\nabla f(\xx^k)}{\xx^k-\xx_k^i}\\
        &\quad + \inp{\nabla f_j(\xx^\star) -\nabla f(\xx^\star)}{\xx^k-\xx_k^i}\\
        &\overset{(iii)}{\leq} \norm{\nabla f_j(\xx_k^j) -\nabla f_j(\xx^k)}\norm{\xx_k^i-\xx^k} + \norm{\nabla f_j(\xx^k) -\nabla f_j(\xx^\star)}\norm{\xx_k^i-\xx^k} \\
        &\quad + \norm{\nabla f(\xx^\star) -\nabla f(\xx^k)}\norm{\xx_k^i-\xx^k} + \norm{\nabla f_j(\xx^\star) -\nabla f(\xx^\star)}\norm{\xx_k^i-\xx^k}\\
        &\overset{(iv)}{\leq} \frac{L}{2}\norm{\xx_k^j -\xx^k}^2 + \frac{3L}{2}\norm{\xx_k^i-\xx^k}^2 + \beta_{f_j}(\xx^\star,\xx^k) + \beta_{f}(\xx^\star,\xx^k) \\
        &\quad + \frac{a^q\norm{\nabla f_j(\xx^\star)-\nabla f(\xx^\star)}^q}{q} + \frac{\norm{\xx_k^i-\xx^k}^p}{pa^p},
    \end{align*}
    where in $(iii)$ we used the Cauchy-Schwarz inequality, and in $(iv)$ we used \Cref{assumption:convexity,assumption:smoothness} and Young's inequality (for any $a>0$). The value of $a$ will be chosen later.

    Now summing from $j=0$ to $i-1$ and from $i=0$ to $n-1$, we have:
    \begin{align*}
        &\sum_{i=0}^{n-1}\sum_{j=0}^{i-1}\inp{\gg_k^j - \nabla f(\xx^k)}{\xx^k-\xx_k^i}\\
        \leq & \frac{3(n-1)L}{2}\sum_{i=0}^{n-1} \norm{\xx_k^i-\xx^k}^2 + \frac{1}{pa^p}\sum_{i=0}^{n-1}i\norm{\xx_k^i-\xx^k}^p + \sum_{i=0}^{n-1}(n-1-i)\beta_{f_i}(\xx^\star,\xx^k) \\
        & + \frac{n(n-1)}{2}\beta_f(\xx^\star,\xx^k)+ \frac{a^q}{q}\sum_{i=0}^{n-1}(n-1-i)\norm{\nabla f_i(\xx^\star)-\nabla f(\xx^\star)}^q\\
        \overset{(v)}{\leq} & \frac{3nL}{2}\sum_{i=0}^{n-1} \norm{\xx_k^i-\xx^k}^2  + \frac{n}{pa^p}\sum_{i=0}^{n-1}\norm{\xx_k^i-\xx^k}^p + \frac{3n^2}{2}\beta_f(\xx^\star,\xx^k) + \frac{a^qn}{q}\sum_{i=0}^{n-1}\norm{\nabla f_i(\xx^\star)-\nabla f(\xx^\star)}^q\\
        \overset{(vi)}{\leq} & \frac{3nL}{2}\sum_{i=0}^{n-1} \norm{\xx_k^i-\xx^k}^2  + \frac{n}{pa^p}\sum_{i=0}^{n-1}\norm{\xx_k^i-\xx^k}^p + \frac{3n^2}{2}\beta_f(\xx^\star,\xx^k) + \frac{a^qn^2\sigma_\star^q}{q}\\
        \overset{(vii)}{\leq} & \frac{3nL}{2}\sum_{i=0}^{n-1} \norm{\xx_k^i-\xx^k}^2  + \frac{n}{pa^p}\sum_{i=0}^{n-1}\norm{\xx_k^i-\xx^k}^p + \frac{3n^2}{2}(F(\xx^k)-F^\star) + \frac{a^qn^2\sigma_\star^q}{q}
    \end{align*}
    where in $(v)$ we used the fact that $\sum_{i=0}^{n-1}\beta_{f_i}(\xx^\star,\xx^k)=n\beta_f(\xx^\star,\xx^k)$, and in $(vi)$ we used \Cref{assumption:heavy-tail}. In $(vii)$ we used \Cref{fact:bregman-composite}.

    Therefore,
    \begin{align*}
        & \frac{2\gamma -nL}{2}\sum_{i=0}^{n-1}\norm{\xx_k^i-\xx^k}^2 + \frac{\lambda}{p2^{p-1}}\sum_{i=0}^{n-1}\norm{\xx_k^i-\xx^k}^p \\
        \leq & \frac{n^2}{2} \big(F(\xx^k) - F^\star\big) + \sum_{i=0}^{n-1}\sum_{j=0}^{i-1}\inp{\gg_k^j - \nabla f(\xx^k)}{\xx^k-\xx_k^i} \\
        \leq & \frac{3nL}{2}\sum_{i=0}^{n-1} \norm{\xx_k^i-\xx^k}^2  + \frac{n}{pa^p}\sum_{i=0}^{n-1}\norm{\xx_k^i-\xx^k}^p + 2n^2(F(\xx^k)-F^\star) + \frac{a^qn^2\sigma_\star^q}{q}.
    \end{align*}
    Rearranging the above, we have:
    \[
        2(\gamma-2nL)\sum_{i=0}^{n-1}\norm{\xx_k^i-\xx^k}^2 + \frac{1}{p} \big(\frac{\lambda}{2^{p-1}}-\frac{n}{a^p}\big)\norm{\xx_k^i-\xx^k}^p \leq 2n^2(F(\xx^k)-F^\star) + \frac{a^qn^2\sigma_\star^q}{q}.
    \]
    Now we can choose $a$ such that $a^p=\frac{n2^{p-1}}{\lambda}$, and assuming that $\gamma \geq 4nL$, we then get the desired result.
\end{proof}

Finally we give the proof of \Cref{thm:main-convergence}.

\MainConvergence*
\begin{proof}
    We first restate the following bound when $\gamma \geq 4nL$ (with both sides divided by $T=nK$):
    \[
        \frac{1}{K}\sum_{k=0}^{K-1}\big(F(\xx^{k+1})-F^\star\big) 
        \leq  \frac{2\gamma R_0^2}{Kn} + \frac{4\lambda R_0^p}{pKn}  + \frac{12nL(F(\xx_0)-F^\star)}{\gamma K} + \frac{12n^qL\sigma_\star^q}{q\gamma\lambda^{q-1}}.
    \]
    Following \Cref{lem:stepsize-gamma-lambda}, we write $E=4nL, A = \frac{2R_0^2}{n}, B=\frac{4R_0^p}{pn}, C=\frac{12n^qL\sigma_\star^q}{q}$ and $D=12nL(F(\xx_0)-F^\star)$. Plugging these in, we get the desired result.

    Then we should choose
    \[
        \lambda = \left(\frac{3pn^{q+1}L\sigma^q_{\star}K}{q\gamma R_0^q}\right)^{\frac{1}{q}},\quad \gamma = \max\left\{4nL, \sqrt{\frac{12n^2L(F(\xx_0)-F^\star)}{2R_0^2}}, \left(\frac{12\cdot 2^{q-2}Kn^{q+1}L\sigma_\star^q}{p^{q-1}R_0^q}\right)^{\frac{1}{q+1}}\right\}
    \]
    and the convergence rate becomes:
    \[
        \frac{1}{K}\sum_{k=0}^{K-1}\big(F(\xx^{k+1})-F^\star\big) \leq \frac{8LR_0^2}{K} + \frac{16\sqrt{LR_0^2(F(\xx_0)-F^\star)}}{K} + \left(\frac{4^{q+2}R_0^{q+2}L\sigma^q_\star}{p^{q-1}qK^q}\right)^{\frac{1}{q+1}}
    \]
    Setting each of the three terms to be at most $\frac{\varepsilon}{3}$, we get the desired result.
\end{proof}

\end{document}